\documentclass[11pt,table]{amsart}
\usepackage{amsmath}
\usepackage[margin=1.0in]{geometry}
\usepackage{amscd,amsmath,amsxtra,amsthm,amssymb,stmaryrd,xr,mathrsfs,mathtools,enumerate,commath, comment, enumitem, graphicx}
\usepackage{amsfonts, amssymb, amsthm, amsmath, braket, hyperref, mathtools, comment, mathrsfs, enumitem, cleveref}
\usepackage{stmaryrd}
\usepackage{orcidlink}
\usepackage{multirow}
\usepackage{xcolor}
\usepackage{listings}
\usepackage{commath}
\usepackage{comment}
\usepackage{multicol}
\usepackage{tikz-cd}
\usepackage{tkz-graph}
\usepackage{colonequals}
\usepackage{hyperref}
\usepackage{graphicx}
\usepackage{bigstrut}
\usepackage{longtable, multirow, array}

\newcommand{\Mod}[1]{\ (\mathrm{mod}\ #1)} 

\newcommand{\calf}{\mathcal{F}}

\usepackage{longtable} 
\usepackage{pdflscape} 
\usepackage{booktabs}
\usepackage{hyperref}
\definecolor{vegasgold}{rgb}{0.77, 0.7, 0.35}
\definecolor{darkgoldenrod}{rgb}{0.72, 0.53, 0.04}
\definecolor{gold(metallic)}{rgb}{0.83, 0.69, 0.22}
\hypersetup{
 colorlinks=true,
 linkcolor=darkgoldenrod,
 filecolor=brown,      
 urlcolor=gold(metallic),
 citecolor=darkgoldenrod,
 }

\usepackage[all,cmtip]{xy}
\DeclareFontFamily{U}{wncy}{}
\DeclareFontShape{U}{wncy}{m}{n}{<->wncyr10}{}
\DeclareSymbolFont{mcy}{U}{wncy}{m}{n}
\DeclareMathSymbol{\Sh}{\mathord}{mcy}{"58}
\usepackage[T2A,T1]{fontenc}
\usepackage[OT2,T1]{fontenc}

\usepackage{tikz}
\usetikzlibrary{shapes.geometric}
\usetikzlibrary{decorations.markings}
\tikzset{every loop/.style={min distance=10mm,looseness=10}}
\tikzstyle{vertex}=[auto=left,circle,minimum size=1pt,inner sep=0pt]

\newtheorem{theorem}{Theorem}[section]
\newtheorem{lemma}[theorem]{Lemma}

\newtheorem*{theorem*}{Theorem}
\newtheorem*{ass*}{Assumption}

\newtheorem{corollary}[theorem]{Corollary}
\newtheorem{remark}[theorem]{Remark}
\newtheorem{example}[theorem]{Example}

\newtheorem{conjecture}[theorem]{Conjecture}
\newtheorem{proposition}[theorem]{Proposition}

\renewcommand{\a}{\alpha}
\renewcommand{\b}{\beta}
\newcommand{\rad}{\operatorname{rad}}

\newcommand{\Z}{\mathbb{Z}}

\newcommand{\Q}{\mathbb{Q}}
\newcommand{\F}{\mathbb{F}}
\newcommand{\D}{\mathfrak{D}}

\DeclareMathOperator{\Res}{Res}

\numberwithin{equation}{section}

\definecolor{codegreen}{rgb}{0,0.6,0}
\definecolor{codegray}{rgb}{0.5,0.5,0.5}
\definecolor{codepurple}{rgb}{0.58,0,0.82}
\definecolor{backcolour}{rgb}{0.95,0.95,0.92}

\lstdefinestyle{sageStyle}{
    backgroundcolor=\color{backcolour},   
    commentstyle=\color{codegreen},
    keywordstyle=\color{magenta},
    numberstyle=\tiny\color{codegray},
    stringstyle=\color{codepurple},
    basicstyle=\ttfamily\footnotesize,
    breakatwhitespace=false,         
    breaklines=true,                 
    captionpos=b,                    
    keepspaces=true,                 
    numbers=left,                    
    numbersep=5pt,                  
    showspaces=false,                
    showstringspaces=false,
    showtabs=false,                  
    tabsize=4,
    frame=single, 
    language=Python 
}

\begin{document}

\title[Monogenic Fields from Polynomial Compositions] {Monogenic Fields from Polynomial Compositions with Applications}

\author[A. Jakhar]{Anuj Jakhar\, \orcidlink{0009-0007-5951-2261}}
\address[Jakhar]{Department of Mathematics, Indian Institute of Technology Madras, Chennai, Tamil Nadu, India-600036}
\email{anujjakhar@iitm.ac.in}

\author[R. Kalwaniya]{Ravi Kalwaniya$^*$\, \orcidlink{0009-0008-6964-5276}}
\address[Kalwaniya]{Department of Mathematics, Indian Institute of Technology Madras, Chennai, Tamil Nadu, India-600036}
\email{ravikalwaniya3@gmail.com; ma22d021@smail.iitm.ac.in}
\thanks{$^*$Corresponding author. Email: ma22d021@smail.iitm.ac.in}

\author[P. Yadav]{Prabhakar Yadav\, \orcidlink{0009-0000-9622-3775}}
\address[Yadav]{Department of Computer Science, Ashoka University, Sonipat, Haryana, India - 131029}
\email{pkyadav914@gmail.com; yadavprabhakar096@gmail.com}

\maketitle
\begin{flushright}
    {\it{-- -- Dedicated to Professor K\'alm\'an Gy\H{o}ry}}
\end{flushright} 
\begin{abstract}
A number field $K$ is called \emph{monogenic} if its ring of integers $\mathbb{Z}_K$ can be expressed as a simple ring extension $\mathbb{Z}[\alpha]$ for some $\alpha \in \mathbb{Z}_K$. A monic irreducible polynomial $f(x)\in\mathbb{Z}[x]$ is said to be monogenic if one of its roots generates both the number field and its ring of integers. In this article, we establish the necessary and sufficient conditions for $[\mathbb{Z}_{K_i}:\mathbb{Z}[\alpha_i]]=1$, where $K_i=\mathbb{Q}(\alpha_i)$ and $\alpha_i$ is a root of the composed polynomial $f_i(x^k+b)$ for $i=1,2$. Here, $f_1(x)=x^n+c\sum_{j=1}^{n}(ax)^{n-j}\in\mathbb{Z}[x]$ and $f_2(x)=x^n+c\sum_{j=1}^{n}a^{j-1}x^{n-j}\in\mathbb{Z}[x]$ are irreducible polynomials of degree $n\ge 3$.
In addition, we derive asymptotic estimates for the number of monogenic polynomials in these families under natural assumptions. As an application of our main results, we construct a class of polynomials with non-square-free discriminants. We also analyze the behavior of solutions to certain related differential equations.
\end{abstract}
\maketitle

\section{Introduction}
 A number field $K$ is said to be \emph{monogenic} if its ring of integers $\Z_K$ admits a power integral basis; that is, if there exists $\theta \in \Z_K$ such that $\Z_K = \mathbb{Z}[\theta]$. The study of monogenic number fields is a classical topic in algebraic number theory, originating in the work of Dedekind. It continues to attract considerable attention due to its deep connections with discriminants, index forms, and Diophantine equations. Let $f(x)\in\mathbb{Z}[x]$ be an irreducible monic polynomial and let $\alpha$ be a root of $f(x)$. The polynomial $f(x)$ is called \emph{monogenic} if $\Z_{\mathbb{Q}(\alpha)}=\mathbb{Z}[\alpha]$. In this case, the number field $\mathbb{Q}(\alpha)$ is necessarily monogenic. However, in general, the converse does not hold. This distinction highlights the subtle nature of monogenity at the level of polynomials.\\[1mm]
In recent decades, explicit criteria for the monogenity of polynomials
have played a central role in the construction of monogenic fields.
Notable examples include pure fields, trinomials, generalized
trinomials, and families related to special recursive polynomials. A
particularly subtle situation arises when considering 
\emph{compositions of polynomials}. Even if a polynomial $f(x)$ is
monogenic, the composed polynomial $f(g(x))$ does not need to be monogenic and
the arithmetic of the index
$[\Z_K:\mathbb{Z}[\alpha]]$ becomes significantly more involved.

\par A fundamental tool in this context is the relation between the
discriminant of a number field and that of a defining polynomial.
Let $K=\mathbb{Q}(\theta)$ and $f(x)$ be the minimal
polynomial of $\theta$, then
\begin{equation}\label{formula}
	\D_f = [\mathbb{Z}_K : \mathbb{Z}[\theta]]^2 d_K,
\end{equation}
where $d_K$ denotes the discriminant of $K$ and $\D_f$ is the discriminant of $f(x)$\footnote{If $f(x)$ is a monic polynomial of degree $n$ with roots $\theta_1,\dots,\theta_n$, then its discriminant is
$\prod_{1 \le i < j \le n}(\theta_i-\theta_j)^2$.}.
Formula~\eqref{formula} shows that knowledge of $\D_f$ and the index $[\mathbb{Z}_K:\mathbb{Z}[\theta]]$ allows one to determine the field discriminant $d_K$.\\[1mm]
The study of the irreducibility and arithmetic properties of polynomial compositions has a rich history. The work of Jones has provided deep insights
into the irreducibility and Galois-theoretic behavior of composite polynomials. From an arithmetic perspective, Evertse and Gy\H{o}ry \cite{EK}, and further Ga\'al~\cite{Gal} developed general techniques for controlling indices and discriminants, which are particularly effective when combined with congruence conditions in coefficients. These methods have been successfully applied to several families of composed polynomials, although a general theory remains far from complete. Significant progress in the monogenity of polynomial compositions has been achieved in recent years. In particular, Jones~\cite{Jones2022}
constructed infinite families of monic Eisenstein polynomials $f(x)\in\mathbb{Z}[x]$ such that the power composition polynomials $f(x^{d^{n}})$ are monogenic for all integers $n\ge 0$ and any integer $d>1$, provided that $f(x)$ is Eisenstein with respect to every prime
divisor of $d$. Subsequently, Jones~\cite{Jones2023} investigated the monogenity of Shanks polynomials of power composition and later established in~\cite{Jones2024} a necessary and sufficient criterion for the monogenity of $f(x^{p})$ in terms of the periodicity of the associated linear recurrence sequence modulo $p$ and $p^{2}$. Recently, Barman, Narode and Wagh~\cite{BarmanNarodeWagh2026} investigated the monogenity of reciprocal polynomials. They obtained sufficient conditions for even degree reciprocal polynomials to be monogenic, partially proved a conjecture of Jones (2021), and established a lower bound for certain sextic monogenic reciprocal polynomials and, furthermore, they \cite{BNW026} studied the monogenity of a special type family of polynomials. In a joint work with Barman, we \cite{BJKY26}  further generalized this family and characterized its monogenity. These works also produced several infinite families that illustrate the theory.
There has been recent interest in investigating the monogenity of compositional and iterated polynomials (see \cite{Castillo2022, Gaal2, Jakhar2024Binomial,  J-K-Y, Jones2021, KoenigSmithWolske2025, Li2022, SharmaSarmaLaishram2024, Smith23,   hanson}). For a comprehensive overview of problems related to discriminants and monogenity, as well as further developments in the effective and algorithmic theory of monogenic number fields, we refer the reader to the survey articles by Ga\'al~\cite{Gaal2024} and Evertse and Gy\H{o}ry~\cite{EvertseGyory2025}.\\[1mm]
More recently, Kaur, Kumar, and Remete~\cite{KKR} obtained the  necessary and sufficient conditions for a monic irreducible polynomial $f(x)\in\mathbb{Z}[x]$ to have the property that the power-compositional
polynomial $f(x^{k})$, with $k\in\mathbb{Z}_{>0}$, is monogenic. In another direction, Sharma and Sarma~\cite{SharmaSarmaMonogenity} studied
compositions $f\circ g$, where $f(x)=x^{n}+ax+b$ is an irreducible trinomial and $g(x)\in\mathbb{Z}[x]$ is such that $f(g(x))$ remains irreducible. They provided necessary and sufficient conditions for the monogenity of $f\circ g$, based on the arithmetic of primes that divide the discriminant of $f(x)$ and their contribution to the index $[\mathbb{Z}_K:\mathbb{Z}[\alpha]]$.
In a related direction, Harrington and Jones~\cite{HarringtonJones2021} introduced two infinite families of polynomials, namely,
\begin{equation}\label{eqn:1.2}
    f_1(x)=x^{n}+c\sum_{i=1}^{n}(ax)^{\,n-i}
\quad \text{and} \quad
f_2(x)=x^{n}+c\sum_{i=1}^{n}a^{\,i-1}x^{\,n-i},
\end{equation}
and established explicit formulae for their discriminants together with criteria for their irreducibility for some particular choices of $a$ and $c$. Subsequently, we ~\cite{JakharKalwaniyaYadav+2024+1147+1154} investigated the  monogenity of these two families under certain conditions.\\[1mm]
Motivated by recent interest in developing a criterion for the monogenity of compositions, the present article studies the monogenity of polynomial compositions of the form $f_i(g(x))$, where $f_i(x)$ are defined in \eqref{eqn:1.2} and $g(x)=x^k+b\in\mathbb{Z}[x]$ is an arbitrary binomial. In this direction, we provide the following results.
\begin{theorem}\label{main result1}
Let $F(x)=(x^k+b)^n + c\sum_{i=1}^{n}(a(x^k+b))^{n-i}\in\Z[x]$ be an irreducible polynomial of degree $nk$ over $\mathbb{Q}$. Let $K=\mathbb{Q}(\theta)$, where $\theta$ is a root of $F(x)$. Then a prime divisor $p$ of $\D_F$ does not divide the index $[\mathbb{Z}_K : \mathbb{Z}[\theta]]$ if and only if one of the following statements holds:
\begin{itemize}
    \item[\textup{(i)}] If $p\mid c$, then $p^2\nmid c$.
    \item[\textup{(ii)}]  If $p\nmid c$ and $p\mid k$ with $k=p^js'$ such that $p\nmid {s'}$. Then $\frac{1}{p}[F(x)-h(x)^{p^j}]$ and $h(x)$ are coprime modulo $p$, where $h(x)=(x^{s}+b)^n+c\sum_{i=1}^n(a(x^s+b)^{n-i}$. If, in addition, $p\mid a$ and $p\mid n$ with $n=p^is$ such that $p\nmid s$, then $p\nmid [c^{n+1}\,a_1^n+(-c_2)^n],$ where $a_1=\frac{a}{p}$ and $c_2=\frac{c+(-c)^{p^{i+j}}}{p}$.
    
   \item [\textup{(iii)}] If $p\nmid ck,\,p\mid a$ and $p\mid n$ with $n=p^js$ such that $p\nmid s$, then $p\nmid [(-c_2)^s+c(ca_1)^s]$, where $a_1=\frac{a}{p}$ and $c_2=\frac{c+(-c)^{p^j}}{p}$.
   \item [\textup{(iv)}] If $p\nmid ck$ and $p \mid an$ but $p \nmid \gcd(a,n)$, then $p^2 \nmid f_1(b)$ whenever $k \geq 2$.
    \item [\textup{(v)}]If $p\nmid ckan$ and $p\mid (n+1)$, then $p^2\nmid f_1(b),\,p^2\nmid (1-a^nc)$ and $p^2\nmid (1+a^ncn)$. 
    \item[\textup{(vi)}] If $p\nmid ckan(n+1)$, then $p^2\nmid f_1(b),\,p\nmid (1+a^ncn)$ and  $p^2\nmid [ n^n(1-a^nc)^{n+1}+a^nc(n+1)^{n+1}].$
\end{itemize}
\end{theorem}

\textbf{Remark.} Let $f(x)\in\mathbb{Z}[x]$ be an irreducible polynomial with a root $\theta$, and let $K=\mathbb{Q}(\theta)$. Then the localized ring $\mathbb{Z}_{(p)}[\theta]$ is integrally closed in $K$ if and only if $p\nmid[\mathbb{Z}_K:\mathbb{Z}[\theta]]$. The methods developed to prove the above theorem and the subsequent results can be extended to polynomials with coefficients in a general Krull valuation ring by replacing the usual $p$-adic arguments with the corresponding valuation-theoretic ones. Thus, the results above remain valid in this more general setting.

The following corollary follows as a consequence of the above theorem. It will be used in Section \ref{ana} and the proof will be given in the same section.
\begin{corollary} \label{cor:2.2}
    Let $n$ and $k$ be positive integers. Let $a,b$ and $c$ be integers such that $\rad(k) \mid c$, $\rad(n+1) \mid ac$ and $\gcd(b,c) > 1$. Let $t=\gcd(a,n)$ and denote $n_1 = \frac{n}{t}$, $a_1 = \frac{a}{t}$. Define the polynomial $f_1(x) = x^n + c\sum_{i=1}^{n}(ax)^{n-i} \in \Z[x]$. Further, assume that the following conditions hold true:
    \begin{itemize}
        \item[\textup{(i)}] For prime $p \mid t$, if $p^2 \mid a$, then $(-c)^{p} \not\equiv -c \Mod{p^2} $. If $p \parallel a$, then $(-c)^{p} \equiv -c \Mod{p^2}$.
        \item[\textup{(ii)}] $c$, $f_1(b)$ and $n_1^n (1-a^nc)^{n+1} + a_1^n c (n+1)^{n+1}$ are squarefree.
    \end{itemize}
    Then the polynomials $f_1(x)$ and $F(x) = f_1(x^k+b) \in \Z[x]$ are both irreducible and monogenic.
\end{corollary}

\begin{theorem}\label{main result2}
Let $\calf(x)=(x^k+b)^n + c\sum_{i=1}^{n}a^{i-1}\,(x^k+b)^{n-i}\in\Z[x]$ be an irreducible polynomial of degree $nk$ over $\mathbb{Q}$. Let $K=\mathbb{Q}(\theta)$, where $\theta$ is a root of $\calf(x)$. Then a prime divisor $p$ of $\D_{\calf}$ does not divide the index $[\mathbb{Z}_K : \mathbb{Z}[\theta]]$ if and only if one of the following statements holds:
\begin{enumerate}
    \item [\textup{(i)}] If $p\mid c$, then $p^2\nmid a^{n-1}c$.
    \item [\textup{(ii)}] If $p\nmid ca$ and $p\mid k$ with $k=p^js'$ such that $p\nmid {s'}$, then $\frac{1}{p}[\calf (x)-h(x)^{p^j}]$ and $h(x)$ coprime modulo $p$, where $h(x)=(x^{s'}+b)^n+c\sum a^{i-1}(x^{s'}+b)^{n-i}$.
    \item [\textup{(iii)}] If $p\nmid cak$ and $p\mid n$, then $p^2\nmid f_2(b)$, whenever $k \geq 2$.
     \item [\textup{(iv)}] If $p\nmid cak$ and $p\mid(n+1)$, then $p^2\nmid f_2(b),\,p^2\nmid (a+nc)$ and $p^2\nmid (c-a)$.
     \item [\textup{(v)}]If $p\nmid cak(n+1)$, then $p^2\nmid f_2(b),\,p\nmid (nc+a)$ and $p^2\nmid [(-n)^n(c-a)^{n+1} - a^nc~(n+1)^{n+1}]$.
\end{enumerate}
\end{theorem}
\begin{remark}
We wish to point out that, by taking $b=0$ and $k=1$ in the above theorems, we recover the main results of \cite{JakharKalwaniyaYadav+2024+1147+1154}. In addition, no additional conditions are imposed on $a,\,c$ and $n$.
\end{remark}
In 2012, Kedlaya~\cite{Kedlaya2012} provided a method to construct, for each integer $n \geq 2$, infinitely many degree-$n$ monic irreducible polynomials $f(x) \in \mathbb{Z}[x]$ such that $\D_f$ is squarefree. Then these polynomials are necessarily monogenic by~\eqref{formula}. However, when $\D_f$ is not squarefree, this no longer guaranties that $f(x)$ is monogenic, and determining monogenity in such cases can be difficult. 

In 2020, Jones~\cite{Jones2020}, using a result of Pasten~\cite[Theorem~1.1]{Pasten2015} along with a modification of Kedlaya's techniques, constructed explicit infinite families of monic irreducible polynomials $f(x) \in \mathbb{Z}[x]$ of degree $p$, for every prime $p \geq 3$, such that $\D_f$ is not squarefree but $f(x)$ is monogenic. Furthermore, Jones and White~\cite{L3} gave an asymptotic formula for certain special types of monogenic trinomials. In 2022, Bhargava, Shankar and  Wang~\cite{bhargava} proved that at least $30\%$ of polynomials have squarefree discriminant and are therefore monogenic.
In 2025, Barman, Narode and Wagh~\cite{RAV} constructed another infinite family of monogenic polynomials of degree~$q$ with non-squarefree discriminants, where $q$ is a prime of the form $q = q_0 + q_1 - 1$, with $q_0$ and $q_1$ both prime. In this direction, we provide the following result.

\begin{theorem}\label{analytic main result}
    Let $n$ and $k$ be positive integers. Let $\ell$ be a prime and $\kappa$ be the integer $\rad(\ell k)$ and $\omega_c(p)$ denote the number of integers $\alpha$ in the range $1 \leq \alpha \leq p^2$ for which $ f_1(\alpha) \equiv 0 \Mod{p^2}$. If the $abc$-conjecture is true, then for a fixed integer $a$ divisible by $\rad(n+1)$, there are at least
    \[
    \prod_{p \leq \sqrt{n}}\frac{1}{p^2} \left( 1 - \frac{\omega_G(p)}{p^2} \right) \prod_{p > \sqrt{n}} \left( 1 - \frac{n}{p^2} \right)^2 \prod_{p \mid \ell  k} \left( 1-\frac{1}{p} \right) \frac{BC}{\zeta(2) \kappa \varrho^2}
    \] 
    pairs of monogenic polynomials $f_1(x) = x^n + c \sum_{i=1}^{n} (ax)^{n-i} \in \mathbb{Z}[x]$ and $F(x) = (x^k+b)^n + c \sum_{i=1}^{n} (a(x^k+b))^{n-i} \in \mathbb{Z}[x]$ satisfying $b \leq B$ and $c \leq C$.
\end{theorem}
As an application of our main results, we examine a class of linear differential equations naturally associated with polynomials of the form $F(x)$ and $\calf (x)$. The main idea is to connect the analytic behavior of these equations with the arithmetic properties of the corresponding auxiliary polynomial. In particular, when the associated polynomial is monogenic, its roots admit an explicit description in terms of an integral basis. This leads to a unified characterization of exponential solutions of the differential equation. In this direction, we provide the following results.
\begin{theorem}{\label{t1}}
    Let \begin{align}{\label{difft1}}
     \left(\frac{d^k }{dx^k}+b\right)^ny+ca^{n-1} \left(\frac{d^k }{dx^k}+b\right)^{n-1}y+\cdots+ca \left(\frac{d^k }{dx^k}+b\right)+cy=0
    \end{align}
    be a differential equation where $n\geq 1$. Let $F  =(z^k+b)^n + c\sum_{i=1}^{n}(a(z^k+b))^{n-i}$ be the irreducible auxiliary equation of \eqref{difft1}
  with a root $\theta$. If for each prime $p$ that divides the discriminant $D_{\mathcal{F}}$
 of $\mathcal{F}(z)$ that satisfies any of the conditions \textup{(i)} to \textup{(vi)} of Theorem \ref{main result1}, then the general
 solution of the given differential equation \eqref{difft1} is of the form
\begin{align*}
    y(x) = \sum_{i=1}^{nk} \alpha_i\prod_{j=1}^{nk}e^{c_{j-1}^{(i)} \theta^{j-1}x},
\end{align*}
where $c_{j-1}^{(i)}$ are integers and $\alpha_i$ are arbitrary real constants for all $1\leq i,j \leq nk$.
\end{theorem}
\begin{theorem}{\label{t2}}
    Let \begin{align}{\label{diff}}
     \left(\frac{d^k }{dx^k}+b\right)^ny+c \left(\frac{d^k }{dx^k}+b\right)^{n-1}y+\cdots+ca^{n-1} \left(\frac{d^k }{dx^k}+b\right)+ca^ny=0
    \end{align}
    be a differential equation where $n\geq 1$. Let $F  =(z^k+b)^n + c\sum_{i=1}^{n}\,a^{i-1}\,(z^k+b)^{n-i}$ be the irreducible auxiliary equation of \eqref{diff}
  with a root $\theta$. If for each prime $p$ dividing the discriminant $D_{\mathcal{F}}$
 of $\mathcal{F}(z)$ satisfies any one of the conditions \textup{(i)} to \textup{(v)} of Theorem \ref{main result2}, then the general
 solution of the given differential equation \eqref{diff} is of the form
\begin{align*}
    y(x) = \sum_{i=1}^{nk} \alpha_i\prod_{j=1}^{nk}e^{c_{j-1}^{(i)} \theta^{j-1}x},
\end{align*}
where $c_{j-1}^{(i)}$ are integers and $\alpha_i$ are arbitrary real constants for all $1\leq i,j \leq nk$.
\end{theorem}
\noindent \textbf{The structure of the paper is as follows:} This article is organized into six sections, including an introduction. In Section~\ref{pre}, we present some preliminary results and introduce Propositions~\ref{discriminant_of_F} and \ref{discriminant_of second polyn}, along with their proofs, which provide discriminant formulas for $F(x)$ and $\calf(x)$, respectively. Section~\ref{section:2} is devoted to the proofs of Theorems~\ref{main result1} and \ref{main result2}. In Section~\ref{ana}, we prove Theorem~\ref{analytic main result}, which establishes a lower bound for the number of pairs of polynomials $f_1(x)$ and $F(x)$ such that both are monogenic. Section~\ref{section:5}, contains proof of Theorems~\ref{t1} and \ref{t2}. Finally, Section~\ref{examples} contains examples that illustrate our results.

\section{Preliminaries}\label{pre}
In this section, we compute the discriminant of $f_i(g(x))$ for $i=1,2$ and give equivalent conditions for a prime $p$ that does not divide the index. These results will be used in the proof of our main theorems.\\
Let $K$ be a field, and let $f,\,g \in K[x]$ be polynomials of degrees $n$ and $m$, respectively. We write $\ell(f)$ for the leading coefficient of $f$. Assume that both $f$ and $g$ have nonzero discriminants and that they do not have any common roots. Let $\Res(f,g)$ denote the resultant of $f$ and $g$.\footnote{If $\lambda_1,\ldots,\lambda_n$ and $\mu_1,\ldots,\mu_m$ are the roots of $f$ and $g$, respectively, then
\[
\Res(f,\,g)
=
\ell(f)^m \ell(g)^n
\prod_{\substack{1 \le i \le n \\ 1 \le j \le m}}
(\lambda_i-\mu_j).
\]}

We now state a result of Cullinan (see~\cite{cullinan}), which provides a formula for the discriminant of composition $f\circ g$. Although we do not rewrite the proof here, we point out a small correction to the original statement in~\cite{cullinan}: a sign factor $(-1)^{\binom{n}{2}m}$ is missing in the final expression. This issue arises in the transition from the resultant $\Res(f,f')$ to the discriminant $\D_f$.

\begin{lemma}[\cite{cullinan}] \label{3.1:lemma}
With all all notation as above, the discriminant of the composition
$f \circ g$ is given by
\[
\D_{f \circ g} =
(-1)^{\frac{nm(3nm+n-2m-2)}{2}}
\,\ell(f)^{m-1}
\,\ell(g)^{n(\,nm-m-1)}\,
\D_f^{m}
\,\operatorname{Res}(f \circ g,\, g').
\]
\end{lemma}
\begin{proposition}
Let $f_1(x) =x^n+c\sum_{i=1}^n(ax)^{n-i}$ and $g(x) = x^k + b$ be polynomials in $\mathbb{Z}[x]$.  
Let $F(x) = f_1\big(g(x)\big)$ be an irreducible polynomial.  
Then the discriminant $\D_{F}$ of $F$ is given by
\begin{equation} \label{discriminant_of_F}
\D_F =\pm\, k^{nk}\,c^{n-1}\,\left(b^n+c\sum_{i=1}^{n}(ab)^{n-i}\right)^{k-1}\frac{n^n(1-a^nc)^{n+1}+a^nc(n+1)^{n+1}}{(1+a^ncn)^2}
\end{equation}
\end{proposition}
\begin{proof}
Let $F=f_1\circ g$, Then Lemma \ref{3.1:lemma} gives
\begin{align*}
\D_F
&= (-1)^{\frac{nk(3nk+n-2k-2)}{2}}\,
   \D_f^k\, \operatorname{Res}(f\circ g,\, g^{\prime}) \\
&= (-1)^{\frac{nk(3nk+n-2k-2)}{2}}\,
   \D_f^k \, k^{nk} \,(-1)^{nk(k-1)}\,
   \prod_{i=1}^{k-1} f(g(0)).
\end{align*}
Since $g(0)=b$, this simplifies to
\[
\D_F
= (-1)^{\frac{nk(3nk+n-2k-2)}{2}}\,
  \D_f^{\,k} \, k^{nk} \, f(b)^{\,k-1}.
\]
The discriminant of $f_1(x)$ is known explicitly, due to Harrington and Jones~\cite[Theorem 3]{HarringtonJones2021}. Substituting this expression for $\D_f$ into the above formula, we obtain
\[
\D_F
= \pm\,
 k^{nk}\,c^{n-1}\,\left(b^n+c\sum_{i=1}^{n}(ab)^{n-i}\right)^{k-1}\,\frac{[n^n(1-a^nc)^{n+1}+a^nc(n+1)^{n+1}]}{(1+a^ncn)^2}.
\]
This completes the proof.
\end{proof}

\begin{proposition}
Let $f_2(x) =x^n+c\sum_{i=1}^na^{i-1}\,x^{n-i}$ and $g(x) = x^k + b$ be polynomials in $\mathbb{Z}[x]$.  
Let $\calf(x) = f_2\big(g(x)\big)$ be an irreducible polynomial.  
Then the discriminant $\D_{\calf}$ of $\calf$ is given by
\begin{equation} \label{discriminant_of second polyn}
\D_{\calf} =\pm k^{nk}\,a^{(n-1)(n-2)}\, c^{n-1}\,\left(b^n+c\sum_{i=1}^{n}a^i\,b^{n-i}\right)^{k-1}\frac{[(-n)^n(c-a)^{n+1} - a^nc~(n+1)^{n+1}]}{(nc+a)^2}
\end{equation}
\end{proposition}
\begin{proof}
 Using an argument similar to the above proposition, we get the desired results.
\end{proof}

We recall the well-known Dedekind criterion, which will be used throughout this paper. The equivalence of statements~(i) and~(ii) in Lemma~\ref{dedekind} goes back to Dedekind (see, for example, \cite[Theorem~6.1.4]{HC} and \cite{RD}). The equivalence of~(i) and~(iii) was later established by Uchida~\cite{Uchida77}.

\begin{lemma}\label{dedekind}
Let $f(x) \in \mathbb{Z}[x]$ be an irreducible monic polynomial having factorization $\overline{f}(x) = \overline{g}_1(x)^{e_1} \cdots \overline{g}_{t}(x)^{e_{t}}$ modulo a prime $p$, where $\overline{g}_i(x)$ are distinct irreducible polynomials over $\mathbb{Z}/p\mathbb{Z}$, and $g_i(x) \in \mathbb{Z}[x]$ are their monic lifts.  Let $K = \mathbb{Q}(\theta)$ with $\theta$ a root of $f(x)$. Then the following statements are equivalent:
\begin{itemize}
    \item[\textup{(i)}] $p \nmid [\mathbb{Z}_K : \mathbb{Z}[\theta]]$.
    \item[\textup{(ii)}] For each $i$, either $e_i = 1$ or $\overline{g}_i(x) \nmid \overline{M}(x)$,  
    where
    \[
    M(x) = \frac{1}{p}\left(f(x) - g_1(x)^{e_1} \cdots g_t(x)^{e_t}\right).
    \]
    \item[\textup{(iii)}] $f(x) \notin \langle p, g_i(x) \rangle^2$ in $\mathbb{Z}[x]$ for any $i$, $1 \leq i \leq t$.
\end{itemize}
\end{lemma}

\begin{proposition}
Let $\calf(x)=(x^k+b)^n + c\sum_{i=1}^{n}a^{i-1}\,(x^k+b)^{n-i}\in\Z[x]$ be an irreducible polynomial of degree $nk$ over $\mathbb{Q}$. Let $K=\mathbb{Q}(\theta)$, where $\theta$ is a root of $\calf(x)$. Suppose that $p$ is a prime such that $p\mid a$ and $p\nmid c$. Then $\calf(x)$ is non-monogenic.
\end{proposition}
\begin{proof}
Suppose $p\nmid c$ and $p\mid a$. Then $\calf (x)\equiv(x^k+b)^{n-1}(x^k+b+c)\Mod{p}$. Let $\beta$ be any root of $x^k+b$ modulo $p$. Then keeping in mind $n\geq 3$, we conclude that $\beta$ is a repeated root of $\calf (x)$ modulo $p$. Using $p\mid a$, it is easy to verify $\calf(x)\in\langle p,\,x-\beta\rangle^2$. Then, by Lemma~\ref{dedekind}, $\calf (x)$ is non-monogenic.
\end{proof}
The proof of Theorem~\ref{analytic main result} relies on the $abc$-conjecture, which we recall in the following for completeness.
\begin{conjecture}[abc-conjecture]
Let $a,b,c$ be pairwise coprime positive integers satisfying $a+b=c$. For a nonzero integer $n$, define the radical of $n$ by $\operatorname{rad}(n)=\prod_{p\mid n} p,$ where the product runs over all distinct prime divisors $p$ of $n$.  Then, for every $\varepsilon>0$, there exists a constant $K_\varepsilon>0$ such that $c \leq K_\varepsilon \,\operatorname{rad}(abc)^{\,1+\varepsilon}.$
\end{conjecture}
\section{Proof of Theorems~\ref{main result1} and \ref{main result2}}\label{section:2}

\begin{proof}
$\textbf{Case (i).}$ Suppose that $p \mid c$. Then $ F(x) \equiv (x^k + b)^n \pmod p.$ Let $g(x)$ be any irreducible factor of $x^k + b$ modulo $p$. Since $n \geq 2$, each of such factors appears with multiplicity at least $2$ in $F(x)$ modulo $p$. It follows that $ F(x) \in \langle p, g(x) \rangle^2$ if and only if $ p^2 \nmid c.$ Therefore, by Lemma~\ref{dedekind}, we conclude that $ p \nmid [\mathbb{Z}_K : \mathbb{Z}[\theta]]$ if and only if $ p^2 \nmid c.$\\[1mm]
$\textbf{Case (ii).}$ Suppose $p\nmid c$ and $p\mid k$ with $k=p^is'$ such that $p\nmid s$. Then \[F(x)\equiv\left( (x^{s}+b)^n+c\sum_{i=1}^n(a(x^s+b)^{n-i}\right)^{p^{\ell}}\pmod p.\] Write $h(x)=(x^{s}+b)^n+c\sum_{i=1}^n(a(x^s+b)^{n-i}$. Let $g(x)$ be any factor of $h(x)$ modulo $p$. Then $F(x)\in\langle p, g(x)\rangle^2$ if and only if $\frac{1}{p}[F(x)-h(x)^{p^j}]$ and $h(x)$ have a common root modulo $p$. The above conditions seems to be a non-trivial task for checking, but with some restrictions, the above conditions can be simplified as discussed below. Suppose $p\mid a$ and $p\mid n$ with $n=p^js$ such that $p\nmid s$. In this case, we have $F(x)\equiv\left((x^{s'}+b)^s+c\right)^{p^{i+j}}\pmod p$. Let $h(x)=(x^{s'}+b)^s+c$, then $F(x)\equiv h(x)^{p^{i+j}}\pmod{p}$. Write $h(x)=g_1^{e_1}(x)\cdots g_d^{e_d}(x)+pH(x)$, where $g_1(x),\ldots,g_d(x)$ are  monic polynomials that are distinct as well as irreducible modulo $p$ and $H(x)\in\Z[x]$. Since $h(x)=(x^{s'}+b)^s+c$, we have \[(x^{s'} + b)^{\,np^{i}} = (h(x) - c)^{\,p^{i+j}}.\] Using the Binomial expansion with respect to power $p^{i}$ on the left hand side of the above equation, we have $\big(x^k + b^{p^{i}} + pH_1(x)\big)^n = (h(x) - c)^{p^{i+j}},$ for some polynomial $H_1(x) \in \mathbb{Z}[x]$.
Since $p \mid (b^{\,p^{i}} - b)$, we may rewrite the left-hand side and, applying the Binomial expansion again with respect to the power $n$, we obtain
\[(x^k + b)^n =h(x)^{\,p^{i+j}}+ph(x)H_2(x)+p^2H_3(x)+(-c)^{\,p^{i+j}},\] 
for some polynomials $H_2(x),\, H_3(x) \in \mathbb{Z}[x]$. Using the above identity and the fact that $p \mid a$, it follows that 
\begin{equation}\label{eq:2.1}
    F(x) = h(x)^{\,p^{i+j}}+ph(x)H_4(x)+p^2H_5(x)+ca(x^k+b)+c+(-c)^{\,p^{i+j}},
\end{equation}
 for some polynomials $H_4(x), H_5(x) \in \mathbb{Z}[x]$. As $
 i+j\geq 1$, the first three  summands on the right-hand side of \eqref{eq:2.1} belong to $\langle p,\, g_i(x) \rangle^{2}$ for each $i$, $1\leq i\leq d$. Therefore, for any $i \in \{1,2, \ldots,d\}$, $F(x)\in\langle p,\, g_i(x) \rangle^{2}$ if and only if the polynomial $ca(x^k+b)+c+(-c)^{\,p^{i+j}}$ does so,  that is, the polynomials $\frac{1}{p}\left[ca(x^m+b)+c+(-c)^{\,p^{i+j}}\right]$ and $(x^k+b)^n+c$ have a common root modulo $p$. Write $a=pa_1$ and $c+(-c)^{\,p^{i+j}}=pc_2$. Then, it is equivalent to checking that the polynomials $ca_1(x^k+b)+c_2$ and $(x^k+b)^n+c$ have a common root modulo $p$. These polynomials have a common root modulo $p$ if and only if $p$ divides $[c^{n+1}a_1^n+(-c_2)^n]$. Therefore, by Lemma~\ref{dedekind}, $p\nmid [\Z_K:\Z[\theta]]$ if and only if $p$ does not divide $[c^{n+1}a_1^n+(-c_2)^n]$.\\[1mm]
\textbf{Case (iii).} Suppose $p\nmid ck,\,p\mid a$ and $p\mid n$. Write $n=p^js$, where $p\nmid s$. Then $F(x)\equiv(x^k+b)^{s}+c\pmod p$. Denote $h(x) = (x^k+b)^s + c.$ Then  $F(x) \equiv h(x)^{\,p^{j}} \pmod p.$ Now write $h(x) \equiv g_1(x)^{e_1}\cdots g_d(x)^{e_d} \pmod p,$ so that $h(x) = g_1(x)^{e_1} \cdots g_d(x)^{e_d} + pH(x),$ where $g_1(x), \ldots, g_d(x)$ are distinct monic irreducible polynomials modulo $p$, and $H(x) \in \mathbb{Z}[x]$. Using the expression of  $h(x)$ and Binomial expansion, we have $(x^k+b)^n=h(x)^{\,p^{j}}+(-c)^{\,p^{j}}+ph(x)H_1(x)$ for some polynomial $H_1(x)\in\Z[x]$. Now using the expression of $(x^k+b)^s$ and keeping in mind  $p\mid a$, we obtain
\begin{equation*}\label{eq:2.2}
   F(x)=h(x)^{\,p^{j}}+ph(x)H_1(x)+p^2H_3(x)+ac(x^k+b)+c+(-c)^{p^{j}}. 
\end{equation*}
 As  $j\geq 1$, the first three  summands on the right-hand side of the above equation belong to $\langle p,\, g_i(x) \rangle^{2}$ for each $i$, $1\leq i\leq d$. Therefore, for any $i \in \{1,2, \ldots,d\}$, $F(x)\in\langle p,\, g_i(x) \rangle^{2}$ if and only if the polynomial $ac(x^k+b)+c+(-c)^{p^{j}}$ does so,  that is, the polynomials $\frac{1}{p}\left[ac(x^k+b)+c+(-c)^{p^{j}}\right]$ and $(x^k+b)^{s}+c$ have a common root modulo $p$. Write $a=pa_1$ and $c+(-c)^{\,p^j}=pc_2$. Then, it is equivalent to checking that the polynomials $ca_1x^k+c_2$ and $x^{ks}+c$ have a common root modulo $p$. It is easy to verify that the above polynomials have a common root modulo $p$ if and only if $p\mid [(-c_2)^s+c(ca_1)^s]$. Therefore, by Lemma ~\ref{dedekind}, we have $p\nmid [\Z_K:\Z[\theta]]$ if and only if $p \nmid [(-c_2)^s+c(ca_1)^s]$.\\[1mm]
\textbf{Case (iv).} Suppose  $p\nmid ck$. There are two possible choices for $a$ and $n$: $p\mid a$ and $p\nmid n$, or $p\nmid a$ and $p\mid n$.\\[0.5mm]
$\textbf{Subcase (a).}$ Let $p\mid a$ and $p\nmid n$. Then $F(x)\equiv (x^k+b)^n+c \Mod{p}$ and $F'(x)\equiv nkx^{k-1}(x^k+b)^{n-1}\Mod{p}$. Observe that if $\alpha$ is a root of $F(x)$ modulo $p$ and satisfies $(\alpha^k+b)\equiv 0 \Mod{p}$, then we must have $c\equiv 0 \Mod{p}$, which is not possible. Further, since $p\nmid nk$, we note that $\alpha \equiv 0 \Mod{p}$ is the only possible repeated root of $F(x)$ modulo $p$. It is easy to verify that $\a \equiv 0 \Mod{p}$ of $\bar{F}(x)$ if and only if $p\mid F(0)=f_1(b)$ and $k\geq 2$, that is, in this situation $x$ is a repeated factor of $F(x)$ modulo $p$. A direct computation shows that $F(x)\in \langle p,x\rangle^2$ if and only if $p^2\mid f_1(b)$.\\[0.5mm]
$\textbf{Subcase (b).}$ Let $p\nmid a$ and $p\mid n$. Let $\a$ be a repeated root of $F(x)$ modulo $p$. Then

\begin{align}
    F(\alpha) &=(\a^k+b)^{n}+c\sum_{i=1}^n(a(\alpha^k+b))^{n-i} \equiv 0 \Mod{p}, \label{eqn:2.3} \\
    F'(\alpha) &= k\alpha^{k-1} \left(n(\alpha^k+b)^{n-1}+c\sum_{i=1}^{n-1}(n-i)a^{n-i}(\alpha^k+b)^{n-i-1}\right)\equiv0\Mod{p}. \label{eqn:2.4}
\end{align}
In view of $\eqref{eqn:2.3}$ and $\eqref{eqn:2.4}$, we have
\begin{align}
    [a(\a^k+b)-1]\,F(\a)&\equiv a(\a^k+b)^{n+1}-(a^nc-1)(\a^k+b)^n+c, \label{eqn:2.5}\\
    [a(\a^k+b)-1]\,F'(\a)&\equiv
    k\a^{k-1}(\a^k+b)^{n-1}[(n+1)a(\a^k+b)+n(a^nc-1)]\Mod{p}.\label{eqn:2.6}
\end{align} 
Observe that if $\a$ satisfies $(\alpha^k+b)\equiv 0 \Mod{p}$. Then, as above, we obtain $c\equiv 0 \Mod{p}$, which is not possible. Using $p \nmid ka$ along with $p\mid n$, we see that $\a$ satisfies $\a (n+1)a(\a^k+b) \equiv 0 \Mod{p}$. Note that $\alpha \equiv 0 \Mod{p}$ is the only possible repeated root of $F(x)$ modulo $p$. It is easy to verify that $\a \equiv 0 \Mod{p}$ is a repeated root of $F(x)$ modulo $p$ if  and only if $p $ divides $ F(0) = f_1(b)$ and $k \geq 2$, i.e., in this case $x$ is a repeated factor of $F(x)$ modulo $p$. A simple calculation yields $F(x) \in \langle p,x\rangle^2$ if and only if $p^2 \mid f_1(b)$.\\[1mm]
\textbf{Case (v).} Suppose $p\nmid ckan$ and $p\mid (n+1)$. Let $\a$ be a repeated root of $F(x)$ modulo $p$. Then observe that if $\a$ satisfies $(\a^k+b) \equiv 0 \Mod{p}$ then in view of \eqref{eqn:2.3}, we have $c \equiv 0 \Mod{p}$. This is not possible. Noting this and using $p \nmid k$, we obtain that $\a$ satisfies $\a [(n+1)a(\a^k+b)+n(a^nc-1)] \equiv 0 \Mod{p}$.

Note that $\a \equiv 0 \Mod{p}$ is a repeated root of $F(x)$ modulo $p$ if and only if $p $ divides $ F(0) = f_1(b)$ and $k \geq 2$, i.e., in this case $x$ is a repeated factor of $F(x)$ modulo $p$. A simple calculation yields $F(x) \in \langle p,x\rangle^2$ if and only if $p^2 \mid f_1(b)$. Furthermore, note that $p\mid (n+1)$, and hence from \eqref{eqn:2.6}, we have $p\mid (a^n c-1)$. If $p\mid (a^n c-1)$, then from \eqref{eqn:2.5}, we obtain $a(\alpha^k+b)^{n+1}\equiv -c\Mod{p}$. Since $a\in\F_p^{\times}$, there exists a nonzero integer $\lambda$ such that $a^{-1}\equiv \lambda\Mod{p}$. Thus, we have $(\alpha^k+b)^{n+1}\equiv -\lambda c\Mod{p}$. Assume that $g(x)$ is any repeated factor of $F(x)$ modulo $p$. Then, from the above, $g(x)$ divides $(x^k+b)^{n+1}+\lambda c$ modulo $p$. We now determine the possible choices for $g(x)$. From \eqref{eqn:2.5}, it follows that $(x^k+b)-\lambda$ divides $(x^k+b)^{n+1}+\lambda c$ modulo $p$.

If $g(x)=(x^k+b)-\lambda$ is a repeated factor, then a direct computation shows that $F(x)\in\langle p,g(x)\rangle^2$ if and only if $p^2\mid (1+a^n c n)$.

If $(x^k+b)-\lambda$ is not a repeated factor of $F(x)$ modulo $p$, then $g(x)$ divides $(x^k+b)^{n+1}+\lambda c$ modulo $p$ and $g(x)\not\equiv (x^k+b)-\lambda$. Using the identity
\[
(a(x^k+b)-1)F(x)=a(x^k+b)^{n+1} + (a^n c-1)(x^k+b)^n-c,
\]
we observe that $(a(x^k+b)-1)F(x)\in\langle p,g(x)^2\rangle$ and $(a(x^k+b)-1)F(x)\in\langle p^2,g(x)\rangle$ if and only if $F(x)\in\langle p^2,g(x)\rangle$. Moreover, $(a(x^k+b)-1)F(x)\in\langle p^2,g(x)\rangle$ if and only if $p^2\mid (1-a^n c)$.

Finally, using the identity $\langle p,g(x)\rangle^2=\langle p^2,g(x)\rangle\cap\langle p,g(x)^2\rangle$ and the above arguments, from Lemma~\ref{dedekind}, we conclude that $p\nmid [\Z_K:\Z[\theta]]$ if and only if $p^2\nmid (1-a^n c)$ and $p^2\nmid (1+a^n c n)$.\\[1mm]
\textbf{Case (vi).} Now we consider the last case when $p\nmid cka(n+1)$. Now we show that there exists an integer $\beta$ such that the polynomial $x^k + b -\beta$ is a repeated factor of $F(x)$ over $\mathbb{Z}/p\mathbb{Z}$.\\		
Let $t$ be the maximum power of $p$ dividing $1+a^ncn$ and $\alpha$ be any repeated root of $F(x)$ in the algebraic closure of $\mathbb{Z}/p\mathbb{Z}$. Then using \eqref{eqn:2.6}, we have
$$ k\a (\a^k+b) \left[(n+1)a(\a^k+b)+n(a^nc-1)\right] \equiv 0 \Mod{p}. $$
Observe that if $\a$ satisfies $(\a^k+b) \equiv 0 \Mod{p}$ then in view of \eqref{eqn:2.3}, we have $c \equiv 0 \Mod{p}$. This is not possible. Noting this and using $p \nmid k$, we obtain that $\a$ satisfies $\a [(n+1)a(\a^k+b)+n(a^nc-1)] \equiv 0 \Mod{p}$.

Note that $\a \equiv 0 \Mod{p}$ is a repeated root of $F(x)$ modulo $p$ if and only if $p $ divides $ F(0) = f_1(b)$ and $k \geq 2$, i.e., in this case $x$ is a repeated factor of $F(x)$ modulo $p$. A simple calculation yields $F(x) \in \langle p,x\rangle^2$ if and only if $p^2 \mid f_1(b)$. 

Now for any non-zero repeated root $\a$ of $\bar{F}(x)$, we have $(n+1)a(\a^k+b)+n(a^nc-1) \equiv 0 \Mod{p}$. Using $p\nmid a(n+1)$, we obtain $\alpha^k+b \equiv \frac{-n(a^nc-1)}{a(n+1)} \Mod{p}$. Therefore, there exists an integer in $\mathbb{Z}/p\mathbb{Z}$, say $\beta$ such that		
\begin{equation}\label{eqn:4.7}
    	\beta \equiv \frac{-n(a^nc-1)}{a(n+1)}\Mod{p^{2t+2}}.
\end{equation} 
Thus we have proved that any non-zero repeated root of $\bar{F}(x)$ is a root of $x^k + \bar{b} - \bar{\beta}$. We now show that any $\a_1 \not\equiv 0 \Mod{p}$ satisfying $\a_1^k + b \equiv \beta \Mod{p^{2t+2}}$ is a repeated root of $F(x)$ modulo $p$ if and only if $p^{2t+1}$ divides $ n^n(1-a^nc)^{n+1}+a^nc(n+1)^{n+1} $.
		
Using the above equation, we obtain the following.
\begin{align}\label{eqn:4.8}
	[a(\a_1^k + b)-1] \equiv \frac{n(1-a^nc)}{n+1}-1 \equiv \frac{-(1+a^ncn)}{n+1} \Mod{p^{2t+2}}.
\end{align}
Since $p^t\mid\mid(1+a^ncn)$, in view of the above equation it follows that  $p^t\mid\mid(a(\a_1^k + b)-1)$.


A simple calculation yields $(a(x^k+b)-1)F(x) = ax^{n+1} + (a^nc-1) x^n -c$. Substituting $x= \a_1$, we obtain the following.
\begin{align} \label{eqn:4.9}
	[a(\a_1^k+b)-1]\,F(\a_1)
	\equiv& ~a\left(\frac{n(1-a^nc)}{a(n+1)}\right)^{n+1} +\left(\frac{n(1-a^nc)}{a(n+1)}\right)^n(a^nc-1)-c \Mod{p^{2t+2}}\nonumber \\
	\equiv&\frac{-1}{a^n(n+1)^{n+1}}\,[n^n(1-a^nc)^{n+1} +a^nc(n+1)^{n+1}] \Mod{p^{2t+2}}.
\end{align}
In veiw of \eqref{discriminant_of_F}, note that $p^t$ divides $1+a^ncn$ imply that $p^{2t}$ divides $\left(n^n(1-a^nc)^{n+1}+a^n(n+1)^{n+1}\right)$.
Using this and $p^t\mid\mid (a(\a_1^k + b)-1)$ in \eqref{eqn:4.9}, we obtain $F(\a_1)\equiv 0 \Mod{p^{t}}$. An easy manipulation yields the following equality
\begin{equation} \label{eqn:4.10}
    [ a (x^k+b)-1]\, F'(x) = -akx^{k-1} F(x) + kx^{k-1} (x^k+b)^{n-1} [ a(n+1)(x^k+b) + n (a^nc-1) ].
\end{equation}

Substituting $x=\a_1$ and using \eqref{eqn:4.7} and $\a_1^k + b \equiv \b \Mod{p^{2t+2}}$, we obtain the following.
\begin{equation*}
    [ a (\a_1^k+b)-1 ]\, F'(\a_1) \equiv -ak\a_1^{k-1} F(\a_1) \Mod{p^{t+1}}.
\end{equation*}
Using the above equation and the fact that $p^t \mid\mid \left( a (\a_1^k+b)-1 \right)$, we deduce $F'(\a_1)\equiv 0 \Mod{p}$ if and only if $F(\a_1) \equiv 0 \Mod{p^{t+1}}$. In view of \eqref{eqn:4.9}, we have $F(\a_1) \equiv 0 \Mod{p^{t+1}}$ if and only if $p^{2t+1}$ divides $\left(n^n(1-a^nc)^{n+1}+a^n(n+1)^{n+1}\right)$. Combining both, we conclude that $F(x)$ has a non-zero repeated root modulo $p$ if and only if $p^{2t+1}$ divides $\left(n^n(1-a^nc)^{n+1}+a^n(n+1)^{n+1}\right)$.

Thus, we have shown that $x^k+b-\bar{\beta}$ contains all the irreducible repeated factors $\bar{g}(x) \neq x$ of $\bar{F}(x)$. Let $\bar{\a_1}$ be the non-zero repeated root of $\bar{F}(x)$. So, for a positive integer $\a_1$ that satisfies $\a_1^k + b \equiv \b \Mod{p^{2t+2}}$, we can write
\begin{align} \label{eqn:4.11}
    F(x)=(x^k + b -\beta)^2h_1(x)+(x^k + b -\beta) \frac{F'(\a_1)}{k \a_1^{k-1}} + F(\a_1),
\end{align} 
for some polynomial $h_1(x)\in\mathbb{Z}[x]$. Using Lemma \ref{dedekind}, we have $p \nmid [\mathbb{Z}_K:\mathbb{Z}[\theta]]$ if and only if $F(x) \not\in \langle p,x^k+b-\beta \rangle^2$. 

Now, if $t \geq 1$ and $p^{2t+1}$ divides $\left(n^n(1-a^nc)^{n+1}+a^n(n+1)^{n+1}\right)$, then \eqref{eqn:4.9} implies that $F(\a_1) \equiv 0 \Mod{p^2}$. In addition, $F'(\a_1) \equiv 0 \Mod{p}$ and \eqref{eqn:4.11} imply that $F(x) \in \langle p,x^k + b - \beta \rangle^2$ whenever $t \geq 1$. Hence, we have  $t=0$, i.e., $p \nmid (1+a^ncn)$. 
In this situation,  \eqref{eqn:4.11} and $F'(\a_1) \equiv 0 \Mod{p}$ imply that $F(x)$ belongs to $\langle p,x^k+b-\beta \rangle^2$ if and only if $F(\a_1) \equiv 0 \Mod{p^2}$. Using $p \nmid (1+a^ncn)$ and \eqref{eqn:4.8}, we have $p \nmid \left( a(\a_1^k+b) - 1 \right)$. Therefore, $F(\a_1) \equiv 0 \Mod{p^2}$ if and only if $\left( a(\a_1^k+b) - 1 \right) F(\a_1) \equiv 0$ (mod $p^2$). Using \eqref{eqn:4.9} and the fact that $p\nmid a(n+1)$, we obtain $F(x) \not\in \langle p,x^k+b-\beta \rangle^2$ if and only if $[n^n(1-a^nc)^{n+1}+a^nc(n+1)^{n+1}] \not\equiv 0 \Mod{p^2}$. This completes the proof of the theorem.
\end{proof}
\begin{proof}[Proof of Theorem~\ref{main result2}]
$\textbf{Case (i).}$ Suppose that $p \mid c$. Then $ \calf(x) \equiv (x^k + b)^n \pmod p.$ Let $g(x)$ be any irreducible factor of $x^k + b$ modulo $p$. Since $n \geq 2$, each of such factors appears with multiplicity at least $2$ in $\calf(x)$ modulo $p$. It follows that $ \calf(x) \in \langle p,\, g(x) \rangle^2$ if and only if $ p^2 \nmid c.$ Therefore, by Lemma~\ref{dedekind}, we conclude that $ p \nmid [\mathbb{Z}_K : \mathbb{Z}[\theta]]$ if and only if $ p^2 \nmid a^{n-1}c.$\\[1mm]
\textbf{Case (ii).} Suppose that $p\nmid ca$ and $p\mid k$ with $k=p^js'$ are such that $p\nmid s'$. Then \[\calf(x)\equiv \left((x^{s'}+b)^n+c\sum a^{i-1}(x^{s'}+b)^{n-i}\right)^{p^j}\Mod{p}.\] Denotes $h(x)=(x^s+b)^n+c\sum a^{i-1}(x^s+b)^{n-i}$. Now write $h(x) \equiv g_1(x)^{e_1}\cdots g_d(x)^{e_d} \pmod p,$ so that $h(x) = g_1(x)^{e_1} \cdots g_d(x)^{e_d} + pH(x),$ where $g_1(x), \ldots, g_d(x)$ are distinct monic irreducible polynomials modulo $p$, and $H(x) \in \mathbb{Z}[x]$.\\[1mm]
\textbf{Case (iii).} Suppose $p\nmid cka$ and $p\mid n$.  Let $\a$ be a repeated root of $F(x)$ modulo $p$. Then 
\begin{align}
    \calf(\alpha) &=(\a^k+b)^{n}+c\sum_{i=1}^na^{i-1}\,(\alpha^k+b)^{n-i} \equiv 0 \Mod{p} \label{eqn:4.12} \\
    \calf '(\alpha) &= k\alpha^{k-1} \left(n(\alpha^k+b)^{n-1}+c\sum_{i=1}^{n-1}(n-i)a^{i-1}(\alpha^k+b)^{n-i-1}\right)\equiv0\Mod{p}. \label{eqn:4.13}
\end{align}
In view of $\eqref{eqn:2.3}$ and $\eqref{eqn:2.4}$, we have
\begin{align}
    [(\a^k+b)-a]\,\calf(\a)&\equiv (\a^k+b)^{n+1}+(c-a)(\a^k+b)^n-ca^n \Mod{p}\label{eqn:4.14}\\
    [(\a^k+b)-a]\,\calf '(\a)&\equiv
    k\alpha^{k-1}\,(\alpha^k+b)^{n-1}[(n+1)(\alpha^k+b)+n(c-a)]\Mod{p}.\label{eqn:4.15}
\end{align} 
Using $p\nmid k,\,p\mid n$ and \eqref{eqn:4.15}, we have $ \alpha\,(\alpha^k+b)\equiv 0\Mod{p}$. Observe that if $\a$ satisfies $(\a^k+b) \equiv 0 \Mod{p}$ then in view of \eqref{eqn:4.14}, we have $ca^{n-1} \equiv 0 \Mod{p}$. This is not possible. Note that $\a \equiv 0 \Mod{p}$ is a repeated root of $\calf(x)$ modulo $p$ if and only if $p $ divides $ \calf(0) = f_2(b)$ and $k \geq 2$, i.e., in this case $x$ is a repeated factor of $\calf(x)$ modulo $p$. A simple calculation yields $\calf(x) \in \langle p,x\rangle^2$ if and only if $p^2 \mid f_2(b)$.\\[1mm]
\textbf{Case (iv).} Suppose $p\nmid cak$ and $p\mid (n+1)$. Let $\alpha$ be a repeated root of $\calf (x)$ modulo $p$. Observe that if $\a$ satisfies $(\a^k+b) \equiv 0 \Mod{p}$ then in view of \eqref{eqn:4.14}, we have $ca^{n-1} \equiv 0 \Mod{p}$. This is not possible. Noting this and using $p \nmid k$, we obtain that $\a$ satisfies $\a [(n+1)(\alpha^k+b)+n(c-a)] \equiv 0 \Mod{p}$.

Note that $\a \equiv 0 \Mod{p}$ is a repeated root of $\calf(x)$ modulo $p$ if and only if $p $ divides $ \calf(0) = f_2(b)$ and $k \geq 2$, i.e., in this case $x$ is a repeated factor of $\calf(x)$ modulo $p$. A simple calculation yields $\calf(x) \in \langle p,x\rangle^2$ if and only if $p^2 \mid f_2(b)$. Furthermore, note that $p\mid (n+1)$, and hence from \eqref{eqn:4.15}, we have $p\mid (c-a)$. If $p\mid (c-a)$, then from \eqref{eqn:4.14}, we obtain $(\alpha^k+b)^{n+1}\equiv ca^n \Mod{p}$. Assume that $g(x)$ is any repeated factor of $\calf(x)$ modulo $p$. Then, from the above, $g(x)$ divides $(x^k+b)^{n+1}-ca^n$ modulo $p$. \\
If $g(x)=(x^k+b)-a$ is a repeated factor, then using $p\nmid a$,  it is easy to verify that $\calf(x)\in\langle p,g(x)\rangle^2$ if and only if $p^2\mid (a+nc)$. \\
If $(x^k+b)-a$ is not a repeated factor of $F(x)$ modulo $p$, then $g(x)$ divides $(x^k+b)^{n+1}- ca^n$ modulo $p$ and $g(x)\not\equiv (x^k+b)-a$. Using the identity
\[
[(\a^k+b)-a]\,\calf(\a)= (\a^k+b)^{n+1}+(c-a)(\a^k+b)^n-ca^n,
\]
we observe that $[(x^k+b)-a]\,F(x)\in\langle p,\,g(x)^2\rangle$ and $[(x^k+b)-a]\,\calf(x)\in\langle p^2,\,g(x)\rangle$ if and only if $\calf(x)\in\langle p^2,\,g(x)\rangle$. Moreover, $[(x^k+b)-a]\,\calf(x)\in\langle p^2,\,g(x)\rangle$ if and only if $p^2\mid (c-a)$.\\
Finally, using the identity $\langle p,g(x)\rangle^2=\langle p^2,g(x)\rangle\cap\langle p,g(x)^2\rangle$ and the above arguments, from Lemma~\ref{dedekind}, we conclude that $p\nmid [\Z_K:\Z[\theta]]$ if and only if $p^2\nmid (a+nc)$ and $p^2\nmid (c-a)$. \\[1mm]
\textbf{Case (v).} Now we consider the last case when $p\nmid cka(n+1)$. Now we show that there exists an integer $\beta$ such that the polynomial $x^k + b -\beta$ is a repeated factor of $F(x)$ over $\mathbb{Z}/p\mathbb{Z}$.\\
Let $t$ be the maximum power of $p$ dividing $nc+a$ and $\alpha$ be a repeated root of $\calf(x)=(x^k+b)^n+ c\sum_{i=1}^{n}a^{i-1}(x^k+b)^{n-i}$ in the algebraic closure of $\mathbb{Z}/p\mathbb{Z}$. Then in view of \eqref{eqn:4.15}, we have the following. 
\[k\alpha\,(\alpha^k+b)[(n+1)(\alpha^k+b)+n(c-a)]\equiv0\Mod{p}\]
Observe that if $\a$ satisfies $(\a^k+b) \equiv 0 \Mod{p}$ then in view of \eqref{eqn:4.14}, we have $ca^{n-1} \equiv 0 \Mod{p}$. This is not possible. Noting this and using $p \nmid k$, we obtain that $\a$ satisfies $\a [(n+1)(\alpha^k+b)+n(c-a)] \equiv 0 \Mod{p}$. Note that $\a \equiv 0 \Mod{p}$ is a repeated root of $\calf(x)$ modulo $p$ if and only if $p $ divides $ F(0) = f_1(b)$ and $k \geq 2$, i.e., in this case $x$ is a repeated factor of $\calf(x)$ modulo $p$. A simple calculation yields $\calf(x) \in \langle p,\,x\rangle^2$ if and only if $p^2 \mid f_1(b)$. \\
Now for any non-zero repeated root $\a$ of $\calf(x)$ modulo $p$, we have $(n+1)(\alpha^k+b)+n(c-a)]\equiv0\Mod{p}$. Using $p\nmid (n+1)$, we obtain $\alpha^k+b \equiv \frac{-n(c-a)}{(n+1)} \Mod{p}$. Therefore, there exists an integer in $\mathbb{Z}/p\mathbb{Z}$, say $\beta$ such that		
\begin{equation}\label{eqn:4.16}
    	\beta \equiv \frac{-n(c-a)}{(n+1)}\Mod{p^{2t+2}}.
\end{equation} 
Thus we have proved that any non-zero repeated root of $\calf (x)$ is a root of $x^k + b - \beta$ modulo $p$. We now show that any $\a_1 \not\equiv 0 \Mod{p}$ that satisfies $\a_1^k + b \equiv \beta \Mod{p^{2t+2}}$ is a repeated root of $\calf (x)$ modulo $p$ if and only if $p^{2t+1}$ divides $ (-n)^n(c-a)^{n+1} - a^nc~(n+1)^{n+1}$. Using the above relation, we have
\[(\alpha^k+b)-a\equiv-\frac{nc+a}{(n+1)}\Mod{p^{2t+2}}\]
Since $t$ is the highest power of $p$ dividing $(nc+a)$, in view of the above equation, we conclude that $p^t\mid\mid (\alpha^k+b-a)$. Substituting $\a=\a_1$ into \eqref{eqn:4.14}, we obtain
\begin{align*}
    [(\a_1^k+b)-a]\,\calf(\a_1)&\equiv (\a_1^k+b)^{n+1}+(c-a)(\a_1^k+b)^n-ca^n \Mod{p^{2t+2}},\\
    &\equiv\left(-\frac{n(c-a)}{(n+1)}\right)^{n+1}+(c-a)\left(-\frac{n(c-a)}{(n+1)}\right)^n-ca^n\Mod{p^{2t+2}},\\
    &\equiv \frac{1}{(n+1)^{n+1}}[(-n)^n(c-a)^{n+1}-ca^n(n+1)^{n+1}]\Mod{p^{2t+2}}.
\end{align*}
Taking the derivative of $[(x^k+b)-a]\,\calf(x)= (x^k+b)^{n+1}+(c-a)(x^k+b)^n-ca^n$ with respect to $x$, we obtain
\[[(x^k+b)-a]\,\calf (x)=-kx^{k-1}\calf (x)+kx^{k-1}(x^k+b)^{n-1}[(n+1)(x^k+b)+n(c-a)]\]
Substituting $x=\a_1$ and using \eqref{eqn:4.16} and $\a_1^k + b \equiv \b \Mod{p^{2t+2}}$, we obtain the following.
\[[(\a_1^k+b)-a]\,\calf (\a_1)=-kx^{k-1}\calf (\a_1)\Mod{p^{t+1}}]\]
Using the above equation and the fact that $p^t \mid\mid (\a_1^k+b-a)$, we deduce $\calf '(\a_1)\equiv 0 \Mod{p}$ if and only if $F(\a_1) \equiv 0 \Mod{p^{t+1}}$. Using the arguments above, we conclude that $\calf(x)$ has a non-zero repeated root modulo $p$ if and only if $p^{2t+1}$ divides $(-n)^n(c-a)^{n+1} - a^nc~(n+1)^{n+1}$. Thus, we have shown that $x^k+b-\bar{\beta}$ contains all the irreducible repeated factors $g(x) \neq x$ of $\calf(x)$ modulo $p$. Let $\a_1$ be the non-zero repeated root of $\calf(x)$ modulo $p$. So, for a positive integer $\a_1$ that satisfies $\a_1^k + b \equiv \b \Mod{p^{2t+2}}$, we can write

\begin{align} 
    \calf(x)=(x^k + b -\beta)^2h_2(x)+(x^k + b -\beta) \frac{\calf'(\a_1)}{k \a_1^{k-1}} + \calf(\a_1),
\end{align} 
for some polynomial $h_2(x)\in\mathbb{Z}[x]$. Using Lemma \ref{dedekind}, we have $p \nmid [\mathbb{Z}_K:\mathbb{Z}[\theta]]$ if and only if $F(x) \not\in \langle p,x^k+b-\beta \rangle^2$. 
Using an argument similar to Theorem~\ref{main result1} (vi),  we obtain $\calf(x) \not\in \langle p,x^k+b-\beta \rangle^2$ if and only if $[(-n)^n(c-a)^{n+1} - a^nc~(n+1)^{n+1}] \not\equiv 0 \Mod{p^2}$. This completes the proof of the theorem.
\end{proof}

\section{Analytic Results} \label{ana}
The following theorem proved by Granville \cite[Theorem 1]{granville} is the key ingredient in our proof of Theorem \ref{analytic main result}.
\begin{theorem} \label{thm:5.1}
	Suppose that $f(x) \in \Z[x]$, without any repeated roots. Let $ D=\gcd\{f(n): n \in \Z \}$, and select $D'$ as the smallest divisor of $D$ for which $D/D'$ is square-free. If the $abc$-conjecture is true, then 
	\[ |\{ 1 \leq n \leq X : f(n)/D' \text{ is square-free.} \}| \sim \lambda_f X \]
	where $\lambda_f>0$ is a positive constant, which we determine as follows:
	\[ \lambda_f = \prod_{p \text{ prime}} \left( 1- \frac{\omega_f(p)}{p^{\,2+ \mathfrak{q}_p}} \right) \]
	where, for each prime $p$, let $\mathfrak{q}_p$ be the largest power of $p$ that divides $D'$ and let $\omega_f(p)$ denote the number of integers $a$ in the range $1 \leq a \leq p^{\,2+ \mathfrak{q}_p}$ for which $f(a) \equiv 0 \Mod{p^2}$.
\end{theorem}

To prove Theorem \ref{analytic main result}, we will use Corollary \ref{cor:2.2}. We first give a proof of Corollary \ref{cor:2.2}.

\begin{proof}[Proof of Corollary \ref{cor:2.2}]
    Let $\ell$ be a prime dividing $\gcd(b,c)$. In view of $c$ being squarefree, the polynomials $f_1(x)$ and $F(x)$ are $\ell-$Eisenstein. Now we proceed to show that each prime $p$ dividing $\D_F$ satisfies one of the conditions $\textup{(i) - (vi)}$ of Theorem \ref{main result1}. Let $p$ be a prime dividing $\D_F$.
The conditions $\textup{(i)}$ and $\textup{(ii)}$ trivially follows using the facts $c$ is squarefree and $\rad(k) \mid c$. Assume $p \nmid ck$ and $p \mid \gcd(a,n)$. If $p^2 \mid a$, then we have $a_1 = \frac{a}{p} \equiv 0 \Mod{p}$. This implies
    \[ (-c_2)^s + c\,(ca_1)^s \equiv (-\frac{c+ (-c)^{p^j}}{p})^s \equiv \frac{c+ (-c)^{p^j}}{p} \not\equiv 0 \Mod{p}, \]
    where the last relation follows using the assumption on $c$. If $p \parallel a$, i.e., $a_1 \not \equiv 0 \Mod{p}$ , then the assumption on $c$ implies that $c_2 \equiv 0 \Mod{p}$. This implies
    \[ (-c_2)^s + c(ca_1)^s \equiv c^{1+s} a_1 \not\equiv 0 \Mod{p},  \]
    where the last relation follows using $p \nmid ca_1$. Thus $\textup{(iii)}$ of Theorem \ref{main result1} holds. Condition $\textup{(iv)}$ of Theorem \ref{main result1} trivially follows from the fact $f_1(b)$ is sqaurefree. The assumption $\rad(n+1)$   divides $a$ implies that $\textup{(v)}$ is vacuously true. Lastly $\textup {(vi)}$ follows using the assumption $f_1(b)$ and $n_1^n(1-a^nc)^{n+1} + a_1^n c (n+1)^{n+1} $ being squarefree. Thus we conclude that $F(x)$ is monogenic. Applying Thoerem \ref{main result1} with $k=1$ and $b=0$, along with the arguments similar to those used above, the monogeneity of $f_1(x)$ is ascertained.
\end{proof}

\begin{proof}[Proof of Theorem~\ref{analytic main result}]
    To prove this theorem, we use Corollary~\ref{cor:2.2}. We now verify that all the required conditions of the corollary are satisfied. Consider the polynomial $f_1(x) = x^n + c \sum_{i=1}^{n} (ax)^{n-i} \in \Z[x]$, where $c>1$ is a sqaurefree integer. Write $D= \gcd\{f_1(m): m \in \Z \}$. If $D>1$, then there exists a prime divisor $p$ of $D$. Thus, we have $p$ that divides both $f_1(0) = c$ and $f_1(1) = 1 + c\sum_{i=1}^n a^{n-i}$. This implies $p \mid 1$, which is not possible. Hence, we have $D=1$. As $c>1$, let $p$ be a prime divisor of $c$, then it is easy to observe that $f_1(x)$ is $p-$Eisenstein with respect to $p$, i.e., $f_1(x)$ is irreducible. Therefore, $f_1(x)$ does not have any repeated root.
     Assuming $abc$-conjecture, we use Theorem \ref{thm:5.1} to conclude that there are $\sim \prod_{p \text{ prime}} \left( 1- \frac{\omega_c(p)}{p^2} \right) \frac{B}{\ell}$ positive integers $b \leq B$ divisible by $\ell$ such that $f_1(b)$ is square free, where for each prime $p$, $\omega_c(p)$ denotes the number of integers $\alpha$ in the range $1 \leq \alpha \leq p^2$ for which $ f_1(\alpha)= \a^n + c \sum_{i=1}^{n} (a\a)^{n-i} \equiv 0 \Mod{p^2}$. Using $c$ as squarefree, we have $f_1(p^2) \equiv c \not\equiv 0 \Mod{p^2}$ for any prime $p$. In addition, note that the degree of $f_1(x)$ is $n$, so we have $\omega_c(p) \leq \min \{n,p^{2}-1 \}$. Now, let $\gcd(a,n) = t$, $n=n_1 t$, $a= a_1 t$, and $c= \kappa c_1$. Then we can write $$n^n (1-a^nc)^{n+1} + a^n c (n+1)^{n+1} = t^n\, [ n_1^n\, (1-a^nc)^{n+1} + a_1^n \,c\, (n+1)^{n+1} ]. $$
    Define the polynomial $G(x)$ using the following relation,
    \begin{equation} \label{eqn:5.1}
        (1+ n a^n \kappa x)^2\, G(x) = a_1^n (n+1)^{n+1}\, \kappa x + n_1^n\, (1-a^n \kappa x)^{n+1}.
    \end{equation}
    Let $D= \gcd \{ G(m): m \in \Z\}$. We will show that $D$ is sqaurefree and $G(x)$ does not have any repeated root. Firstly, assume that $d$ is a divisor of $D$, then $d$ divides both $G(0) = n_1^n$ and $G(1)$. Further, $G(1)$ divides $ a_1^n (n+1)^{n+1} \kappa + n_1^n (1-a^n \kappa)^{n+1}$. This implies that $d$ divides $a_1 (n+1) \kappa$ and $n_1$. Noting that $\gcd (a_1 (n+1), n_1) = 1$, we conclude $d \mid \kappa$. Thus, we have $D= \kappa$ which is squarefree. So, according to the notation of Theorem \ref{thm:5.1}, we have $D' = 1$. Now we proceed to show that $G(x)$ does not have any repeated root. A simple algebraic manipulation will show that the only repeated factor on the right hand side of \eqref{eqn:5.1} is $(1+na^n\kappa x)$ that precisely has the order $2$. Thus, $G(x)$ does not have any repeated factor. Thus, there are atleast $\prod_{p \text{ prime} } \left( 1 - \frac{\omega_G(p)}{p^2} \right) \frac{C}{\kappa}$ values $c_1 \leq \frac{C}{\kappa}$. Taking $c= \kappa c_1$, we conclude that there are at least $\prod_{p \text{ prime} } \left( 1 - \frac{\omega_G(p)}{p^2} \right) \frac{C}{\kappa}$ values of $c$ such that $a_1^n (n+1)^{n+1} c + n_1^n (1-a^n c)^{n+1}$ is square free. \\[1mm]
 To align with the conditions of Corollary~\ref{cor:2.2}, we require $c$ to be squarefree and to satisfy condition~(i) of the corollary. Writing $c=\kappa c_1$ and noting that $\kappa$ is squarefree, it follows that $c$ is squarefree if and only if $c_1$ is squarefree and $\gcd(c_1,\kappa)=1$. These restrictions will yield that we are left with \[\prod_{p \text{ prime} } \left( 1 - \frac{\omega_G(p)}{p^2} \right) \frac{C\phi(\kappa)}{\zeta(2)\kappa^2}\] choices of $c$. Now for each prime $p$ that divides $t$, if $p^2$ divides $a$, then select $\gamma_p \in (\Z/p\Z)^*$ so that $\gamma_p^p \not\equiv \gamma_p \Mod{p^2}$. If $v_p(a)=1$, then select $\gamma_p \in (\Z/p\Z)^*$ so that $\gamma_p^p \equiv \gamma_p \Mod{p^2}$. Define $\varrho = \rad(t)$, then by the Chinese remainder theorem, there exists $\gamma \in (\Z/\varrho^2\Z)^*$ such that $\gamma \equiv -\gamma_p \Mod{\varrho^2}$ for all prime $p$ dividing $\varrho$. Now, if we select $c \equiv \gamma \Mod{\varrho^2}$, then the conditions given in Corollary \ref{cor:2.2}(i) are satisfied. For such a restriction on $c$, we are left with $\prod_{p \text{ prime} } \left( 1 - \frac{\omega_G(p)}{p^2} \right) \frac{C\phi(\kappa)}{\zeta(2) \kappa^2 \varrho^2}$ choices of $c$.\\[1mm]
 Combining all of the above estimates along with Corollary \ref{cor:2.2} and noting that $\omega_c(p) \leq \min \{n,p^{2}-1 \}$, we have atleast 
    \[
    \prod_{p \leq \sqrt{n}}\frac{1}{p^2} \left( 1 - \frac{\omega_G(p)}{p^2} \right) \prod_{p > \sqrt{n}} \left( 1 - \frac{n}{p^2} \right)^2 \prod_{p \mid \ell  k} \left( 1-\frac{1}{p} \right) \frac{BC}{\zeta(2) \kappa \varrho^2}
    \] 
    choices of pairs $(b,c)$ such that both $f_1(x)$ and $F(x)$ are monogenic.
\end{proof}

\begin{remark}
Using techniques similar to those in Theorem~\ref{analytic main result}, one can derive an analytic lower bound related to the monogenity of $f_2(x)$ and $\calf(x)$.
\end{remark}
\section{Proof of Theorems~\ref{t1} and \ref{t2}}\label{section:5}
\begin{proof}[Proof of Theorem~\ref{t1}]
The given differential equation is \begin{align}{\label{diff1}}
    \left(\frac{d^k }{dx^k}+b\right)^ny+ca^{n-1} \left(\frac{d^k }{dx^k}+b\right)^{n-1}y+\cdots+ca \left(\frac{d^k }{dx^k}+b\right)+cy=0.
    \end{align}
   Observe that $F  =(z^k+b)^n + c\sum_{i=1}^{n}(a(z^k+b))^{n-i}$ is the irreducible auxiliary equation associated with \eqref{diff1}, having root $\theta$. Suppose that every prime $p$ dividing the discriminant $\D_{\mathcal{F}}$ satisfies one of the conditions (i)–(vii) of Theorem \ref{main result1}. Then, by \eqref{formula}, it follows that $\Z_K = \mathbb{Z}[\theta]$ and 
 \[
 \mathbb{Z}[\theta] = \{c_0 + c_1\theta + c_2 \theta^2 +\cdots + c_{kn-1}\theta^{kn-1} \;|\; c_k \in \mathbb{Z} \text{ for all } 1\leq k\leq kn-1     \}.
 \]
Thus, every root of the equation $F(z) = 0$ can be expressed in the form
$$
c_0^{(i)} + c_1^{(i)}\theta + c_2^{(i)}\theta^2 + \cdots + c_{kn-1}^{(i)}\theta^{kn-1},
$$
where each $c_{j-1}^{(i)} \in \mathbb{Z}$ for all $1 \leq i,j \leq kn$. Consequently, the general solution of the differential equation \eqref{diff1} is given by
\begin{align*}
    y(x) = \sum_{i=1}^{kn} \alpha_i\prod_{j=1}^{kn}e^{c_{j-1}^{(i)} \theta^{j-1}x},
\end{align*}
where  $\alpha_i$ is arbitrary real constants for all $1\leq i,j \leq kn$. This completes the proof of Theorem.
\end{proof}
\begin{proof}[Proof of Theorem~\ref{t2}]
   The proof follows along the same lines as the above argument and is therefore omitted.
\end{proof}

\section{Examples of Monogenic Families}\label{examples}

In this section, we provide an infinite family of monogenic polynomials. By specializing the parameters, we can generate various sequences of irreducible polynomials with a trivial index.

\begin{example}
Let $n = 2$, which implies $\text{rad}(n+1) = 3$. We set $a = 1$, so $t = \gcd(1, 2) = 1$ and condition \textup{(i)} of Corollary \ref{cor:2.2} is  satisfied. Let $c = 30$ and $k$ be any integer such that $\text{rad}(k) \mid 30$ (for instance, $k=2^m$ or $k=5^m$). The polynomial $f_1(x)$ is given by:
\[ f_1(x) = x^2 + 30(x+1) = x^2 + 30x + 30. \]
For condition \textup{(ii)}, we check the squarefreeness of the following term:
\[ n_1^n (1-a^nc)^{n+1} + a_1^n c (n+1)^{n+1} = 2^2(1-30)^3 + 30(3)^3 = -96746. \]
Since $-96746 = -2 \times 48373$ is squarefree, condition \textup{(ii)} is satisfied, provided that $f_1(b)$ is squarefree and $\gcd(b, 30) > 1$. Taking $b=30$, we have $f_1(30) = 1830$, which is squarefree. Thus, the polynomials 
\[ f_1(x)= x^2 + 30x + 30 \quad \text{and} \quad F(x) = (x^k + 30)^2 + 30(x^k + 30) + 30 \]
are both irreducible and monogenic. The existence of infinitely many such $b$ follows from the result of Erd\H{o}s \cite{erdos1953}.
\end{example}

\begin{remark}
The existence of infinitely many such integers $b$ follows from the fact that $f_1(x)$ is primitive. Indeed, by a classical result of Erd\H{o}s \cite{erdos1953}, if $f_1(x) \in \mathbb{Z}[x]$ has no fixed square factor, then there are infinitely many integers $b$ for which $f_1(b)$ is squarefree.
\end{remark}

\begin{example}
Let $n=3$, $k=2$, $a=1$, $b=1$, and $c=2$. Then
$
f_2(x)=x^3+2(1+x+x^2),
$
and
$
\calf(x)=f_2(x^2+1)=x^6+5x^4+9x^2+7.
$
We can verify that both $f_2(x)$ and $\calf(x)$ are irreducible over $\mathbb{Q}$. Using Proposition~\ref{discriminant_of second polyn}, we obtain
$
\D_{\calf} = \pm 2^8 \cdot 7 \cdot 11.
$
Thus, the primes dividing $\D_{\calf}$ are $2$, $7$, and $11$. We now verify the conditions of Theorem~\ref{main result2}. In view of \eqref{formula}, it follows that $7$ and $11$ do not divide $[\mathbb{Z}_K:\mathbb{Z}[\theta]]$. It remains to consider $p=2$. Since $2\mid c$ but $2^2\nmid a^{n-1}c=2$, by Theorem~\ref{main result2} \textup{(i)}, we conclude that $2\nmid [\mathbb{Z}_K:\mathbb{Z}[\theta]]$. Therefore, by Theorem~\ref{main result2}, we conclude that $\calf(x)$ is monogenic.
\end{example}
\begin{example}
Consider $n = 3~, k = 2$, $a = 1$, $c = 2$ and $b = 2$. Then $f_1(x) = x^3 + 2x^2 + 2x + 2 \in \Z[x]$ and $F(x)  = x^6 + 8x^4 + 22x^2 + 22\in\Z[x]$. Clearly, $f_1(x)$ and $F(x)$ are irreducible over $\Q$.
Using the formula in \eqref{discriminant_of_F}, we have
$ \D_F =\pm2^9 \cdot 11^2$. We now verify that the prime divisors $p \in \{2, 11\}$ satisfy the conditions of Theorem \ref{main result1}:

\begin{itemize}
    \item For $p = 2$, we have $p \mid c$ since $c=2$. Because $2^2 \nmid 2$, statement \textup{(i)} is satisfied immediately.
    \item For $p = 11$, we note that $p \nmid ckan(n+1)$ because $ckan(n+1) = 2 \cdot 2 \cdot 1 \cdot 3 \cdot 4 = 48$. Therefore, we refer to statement \textup{(vi)}. We check the three required non-divisibility conditions:
    \begin{enumerate}
        \item $p^2 \nmid f_1(b)$: We have $11^2 \nmid 22$, which holds.
        \item $p \nmid (1+a^ncn)$: We have $11 \nmid (1+6) = 7$, which holds.
        \item $p^2 \nmid [n^n(1-a^nc)^{n+1}+a^nc(n+1)^{n+1}]= 539$. We have $11^2 \nmid 539$, which holds.
    \end{enumerate}
    Thus, statement \textup{(vi)} is completely satisfied for $p=11$.
\end{itemize}

Since every prime divisor of $\D_F$ satisfies the necessary criteria, no prime divides the index $[\mathbb{Z}_K : \mathbb{Z}[\theta]]$. Therefore, $\mathbb{Z}_K = \mathbb{Z}[\theta]$ and $F(x) = x^6 + 8x^4 + 22x^2 + 22$ is a monogenic quadrinomial.
\end{example}
\textbf{Acknowledgments.}
The initiative for this project was taken during the visit of the second author to the Indian Statistical Institute, Delhi Centre. The author sincerely thanks Professor Shanta Laishram for his hospitality and for providing a stimulating research environment.  
\section*{Conflict of interest}
The authors declare that they have no known competing financial interests or personal relationships that could have appeared to influence the work reported in this paper.

\section*{Data availability statement}
Data sharing is not applicable to this article, as no datasets were generated or analyzed during the current study.
\bibliographystyle{amsplain}
\bibliography{references}  
\end{document}